\documentclass[11pt]{amsart}
\usepackage{latexsym,color,amsmath,systeme, amsthm,amssymb,amscd,amsfonts}
\usepackage{graphicx}

\usepackage{hyperref}
\usepackage{pgf,tikz,pgfplots, pgfplotstable}
\pgfplotsset{compat=1.14}
\usepackage{float}
\usepackage{mathrsfs}
\usetikzlibrary{arrows}
\usepackage[style=english]{csquotes}
\usepackage{enumerate}
\usepackage{lscape}
\usepackage[width=\textwidth]{caption}
\usepackage{subcaption}
\DeclareCaptionSubType*[arabic]{figure}
\tikzset{
	v1/.style={line width=.5pt,blue!33!black},
	v2/.style={line width=.5pt,blue!66!black},
	v3/.style={line width=.5pt,blue!33},
	v4/.style={line width=.5pt,blue!66},
	v5/.style={line width=.5pt,black}
}

\definecolor{dred}{HTML}{C11B17}
\definecolor{dgreen}{HTML}{41A317}
\definecolor{dblue}{HTML}{0000FF}
\definecolor{brilliantrose}{rgb}{1.0, 0.33, 0.64}

\newcommand{\cemph}[1]{\textcolor{dred}{\emph{#1}}}

\newcommand{\R}{\mathbb R}
\newcommand{\K}{\mathcal K}

\newcommand{\conv}{\mathrm{conv}}
\newcommand{\bd}{\mathrm{bd}}
\newcommand{\pos}{\mathrm{pos}}
\newcommand{\aff}{\mathrm{aff}}

\newcommand{\inter}{\mathrm{int}}

\newcommand{\GH}{{\mathbb{GH}}}

\definecolor{zzttqq}{rgb}{0.6,0.2,0}
			\definecolor{ccqqqq}{rgb}{0.8,0,0}
\definecolor{ffvvqq}{rgb}{1,0.3333333333333333,0}
\definecolor{zzffqq}{rgb}{0.6,1,0}
\definecolor{qqwuqq}{rgb}{0,0.39215686274509803,0}
\definecolor{ffzzqq}{rgb}{1,0.6,0}
\definecolor{ffqqqq}{rgb}{1,0,0}
\definecolor{zzttqq}{rgb}{0.6,0.2,0}
\definecolor{uuuuuu}{rgb}{0.26666666666666666,0.26666666666666666,0.26666666666666666}

\newtheorem{thm}{Theorem}[section]
\newtheorem{lemma}[thm]{Lemma}

\newtheorem{proposition}[thm]{Proposition}
\newtheorem{cor}[thm]{Corollary}

\newtheorem{rmk}[thm]{Remark}

\usetikzlibrary{calc,intersections}
\begin{document}

\title[Golden ratio in comparing symmetrizations]{Relating Symmetrizations of Convex Bodies:\\ Once More the Golden Ratio}

\author[R.~Brandenberg]{Ren\'e Brandenberg}\address{Technical University of Munich, Germany, Department of Mathematics, Discrete Mathematics, Optimization, and Convexity}\email{brandenb@ma.tum.de}
\author[K.~von Dichter]{Katherina von Dichter}\address{Technical University of Munich, Germany, Department of Mathematics, Discrete Mathematics, Optimization, and Convexity}\email{dichter@ma.tum.de}
\author[B.~Gonz\'alez Merino]{Bernardo Gonz\'alez Merino}\address{Departamento de Did\'actica de las Ciencias Matem\'aticas y Sociales, Facultad de Educaci\'on, Universidad de Murcia, 30100-Murcia, Spain}\email{bgmerino@um.es}

\thanks{This research is a result of the activity developed within the framework
of the Programme in Support of Excellence Groups of the Regi\'on de Murcia, Spain, by Fundaci\'on
S\'eneca, Science and Technology Agency of the Regi\'on de Murcia, and
is partially funded by FEDER / Ministerio de Ciencia e Innovaci\'on - Agencia Estatal de Investigaci\'on.
The third author is partially supported by Fundaci\'on S\'eneca project 19901/GERM/15, Spain,
and by MICINN Project PGC2018-094215-B-I00 Spain.}

\begin{abstract}
We show that for any Minkowski centered planar convex compact set $C$
the Harmonic mean of $C$ and $-C$ can be optimally contained in the arithmetic mean of the same sets if and only if the Minkowski asymmetry of $C$ is at most the golden ratio $(1+\sqrt{5})/2 \approx 1.618$.
Moreover, the most asymmetric such set that 
 is (up to a linear transformation) a special pentagon, which we call the golden house.
\end{abstract}

\date{\today}\maketitle

\section{Introduction and Notation}\label{sec:IntrNotMain}
The \cemph{golden ratio} $\varphi=(1+\sqrt{5})/2 \approx 1.618$ has a history of 2400 years and wide roots in Mathematics, Music, Architecture, Biology and Philosophy (see e.g.\ \cite{L}).
It was first studied by the ancient greeks because of its frequent appearance in geometry. For example, if one considers a regular pentagon of edge-length 1, its diagonals have length $\varphi$. No wonder that the regular pentagram was the Pythagorean symbol \cite{L}.
The first known definition is given in Euclid's Elements, II.11: "If a straight line is cut in extreme and mean ratio, then as the whole line is to the greater segment, the greater is to the lesser segment". 
Expressed algebraically this transfers to the (probably) best-known definition of the golden ratio:
\begin{equation}\label{goldenratio}
\text{if } a>b>0 \text{ such that } \frac{a+b}{a} = \frac{a}{b}, \quad \text{then} \quad \frac{a}{b} = \varphi.
\end{equation}

Amongst the fundamental inequalities in mathematics, a special place is reserved for the arithmetic-geometric-harmonic mean inequality, which in the two-argument case together with minimum and maximum states 
\begin{equation}\label{eq:means_of_numbers}
\min\{a,b\}\leq\left(\frac{a^{-1}+b^{-1}}{2}\right)^{-1}\leq \sqrt{ab}\leq \frac{a+b}{2}\leq\max\{a,b\}
\end{equation}
for any real numbers $a,b>0$ (see \cite{HLP,Sch}).
We may identify means of numbers by means of segments by associating $a, b > 0$ with $[-a,a]$ and $[-b,b]$. Doing so we  identify, e.g., the arithmetic mean of $a$ and $b$ with the segment $[-\frac{1}{2} \left( a+b \right), \frac{1}{2} \left( a+b \right) ] =: \frac{1}{2} \left( [-a,a]+[-b,b] \right)$. This way means of convex bodies can be introduced. 

Let $\K^n$ denote the set of \cemph{convex bodies}, i.e.~fulldimensional compact convex sets. For $X\subset\R^n$ let $\conv(X)$ (resp.~$\pos(X)$ or $\aff(X)$) be the \cemph{convex hull} (resp.~ \cemph{positive hull} or \cemph{affine hull}) of $X$, i.e.~the smallest convex set (resp.~convex cone or affine subspace) containing $X$. A \cemph{line-segment} is the convex hull of a two point set $\{x,y\} \subset \R^n$, which we denote by $[x,y]$. For any
$K,C\subset\R^n$, $\rho \in \R$ let $K+C=\{a+b:a\in K,b\in C\}$ be the \cemph{Minkowski sum} of $K, C$ and
$\rho C= \{ \rho x: x \in C\}$ the \cemph{$\rho$-dilatation} of $C$. We abbreviate $(-1)C$ by $-C$.

Now, the \cemph{arithmetic mean} of $K$ and $C$ is defined by $\frac{1}{2} (K+C)$, the \cemph{minimum} by $K\cap C$, and the \cemph{maximum} by $\conv(K\cup C)$. For any $K\in\K^n$ let $K^\circ=\{a\in\R^n: a^T x \leq 1,\, x\in K\}$ be the \cemph{polar} of $K$.
Since the polarity can be regarded as the higher-dimensional replacement of the inversion operation $x\rightarrow 1/x$ (cf.~\cite{MR}), the \cemph{harmonic mean} of $K$ and $C$ is defined by $\left( \frac{1}{2}(K^\circ+C^\circ) \right)^{\circ}$.
The geometric mean has been extended in several ways (cf.~\cite{BLYZ} or \cite{MR}), thus it would need a separate, more involved treatment, which is the reason why we focus on the four other means here. The study of means of convex bodies started in the 1960's  \cite{F,F2,F3}, but there also exist several recent papers 
\cite{MR,MR2,MMR}.

Probably the most essential result of Firey is the extension of the harmonic-arithmetic mean inequality from positive numbers to convex bodies containing 0 in there interior in \cite{F}. Moreover, one can easily show that
Firey's inequality again may be extended involving the minimum and maximum:

\begin{proposition} \label{prop:means_of_sets}
 For all $K,C \in\K^n$ with $0$ in their interior we have
 \begin{equation}\label{eq:means_of_sets}
   K\cap C\subset \left(\frac{K^\circ+C^\circ}{2}\right)^{\circ}\subset\frac{K+C}{2}\subset\conv(K\cup C).
 \end{equation}
\end{proposition}

Let us mention an application given in \cite{F4}. For two positive definite symmetric matrices $A, B \in \R^{n \times n}$ we denote by $A \succcurlyeq B$ if $A-B$ is also positive definite. Since means of ellipsoids correspond to combinations of the corresponding matrices, 
\eqref{eq:means_of_sets} also results in a (generalized) harmonic-arithmetic mean inequality:
\[
(1 - \lambda)A + \lambda B \succcurlyeq ((1 - \lambda)A^{-1} + \lambda B^{-1} )^{-1}
\]
for any $\lambda \in [0,1]$.
The inequality is strict, except in the trivial cases $A=B$ or $\lambda \in \{ 0,1 \}$. Moreover, the well known Brunn-Minkowski determinantal inequality \cite{Ha}
\[((1 - \lambda) \det(A) + \lambda \det(B))^{\frac{1}{n} } \geq  
\det((1 - \lambda)A)^{\frac{1}{n}} + \det(\lambda B)^{\frac{1}{n}},\]\
can be further developed using the means of convex bodies 
as follows \cite{F4}: let $k \in \{1,\dots,n\}$ and $|A|_k$ denote the product of the $k$ greatest eigenvalues of $A$ then
\begin{equation*}\label{prop:bohnenblust}
	|(1 - \lambda)A^{-1} + \lambda B^{-1}|_k^{-\frac{1}{k}}  \leq ((1 - \lambda) |A|_k^{-\frac{1}{k}} + \lambda |B|_k^{-\frac{1}{k}})^{-1}.
\end{equation*}

For any $K,C\in\K^n$ we say that $K$ is \cemph{optimally contained} in $C$, and denote it by $K\subset^{opt}C$, if $K\subset C$ and $K\not \subset t+\rho C$ for any $0 \leq \rho<1$ and $t\in\R^n$.
If $C=t -C$ for some $t \in \R^n$, we say $C$ is \cemph{symmetric}, and if $C=-C$, we say $C$ is \cemph{$0$-symmetric}.
The family of 0-symmetric convex bodies is denoted $\K^n_0$. By $T \in \K^n$ we denote a regular simplex with (bary-)center 0.

The goal of this paper is to consider optimal containments of means of $C$ and $-C$ of a convex body $C$,
i.e.~symmetrizations of $C$. This kinds of symmetrizations are used frequently in convex geometry, e.g.~as extreme cases of a variety of geometric inequalities. Consider, e.g., the Bohnenblust inequality \cite{Bo}, which bounds the ratio of the circumradius and the diameter of convex bodies in arbitrary normed spaces, for which equality is reached in spaces with  $T \cap (-T)$ or $ \frac{1}{2} (T-T)$ as the unit ball \cite{BrK}.  
Or consider the characterization of normed spaces in which $C$ is complete or reduced, if the unit ball is sandwiched between suitable rescalings of two different means of $C$ and $-C$ \cite[Prop.~3.5 to 3.10]{BGJM}.

Also well-known geometric inequalities have been reinvestigated, replacing one mean by another. Consider, e.g.~the Rogers-Shephard-type inequalities, which bound the ratio of the products of the volumes of the maximum and harmonic (resp. arithmetic) means of $K$ and $C$ with the product of their volumes \cite{RS,AGJV,AEFO}.

Notice that for any $C \in \K^n $ we have 
\[
C\cap (-C) \subset^{opt} \left(\frac{C^\circ - C^\circ}{2}\right)^{\circ} \quad \text{and} \quad \frac{C-C}{2} \subset^{opt} \conv(C\cup (-C)). 
\]
Moreover, 
\[\left(\frac{C^\circ-C^\circ}{2}\right)^{\circ}\subset^{opt} \frac{C-C}{2} 
\]
is also possible, i.e.~all containments in \eqref{eq:means_of_sets} may be optimal at the same time 
 even for non-symmetric $C$. Particularly if $T \in \K^3$ is a regular simplex with center $0$ we have the nice situation that the four means are a cross polytope (minimum), a rhombic dodecahedron (harmonic mean), a cube octahedron (arithmetic mean), and a cube (maximum), such that even the cross polytope is optimally contained in the cube.

However, in the planar case, optimal containment of the harmonic mean of $T$ and $-T$ in their arithmetic mean for an equilateral triangle $T$ enforces that the center of the triangle is not 0.
In contrast, for the equilateral triangle $T\subset\R^2$ with center $0$ holds
\[
  \left(\frac{T^\circ-T^\circ}{2}\right)^{\circ}\subset^{opt}\frac89 \cdot \frac{T-T}{2}
   \quad \text{and} \quad T\cap (-T) \subset^{opt}\frac23 \cdot\conv(T\cup (-T)).
\]
Clearly, symmetrizations of an already symmetric $C$ should coincide with $C$, which is always true for the arithmetic mean of $C$ and $-C$, but for the other three considered means only if $0$ is the center of symmetry of $C$. This indicates the need to fix a meaningful center for every convex body first and then concentrate on translates with that center at 0.


Since we want to investigate the optimality of the inequality chain \eqref{eq:means_of_sets} in dependence of asymmetry we will introduce one of the most common asymmetry measures, which fits best to our purposes, and 
choose the center definition matching it.
The \cemph{Minkowski asymmetry} of $C$ is defined by $s(C):=\inf \{ \rho >0: C-c \subset \rho (C-c), c \in \R^n \}$ \cite{Gr} and a \cemph{Minkowski center} of $C$ is any $c \in \R^n$ such that $C-c \subset s(C)(c-C)$ \cite{BG}. Moreover, if $c=0$ is a Minkowski center, we say $C$ is \cemph{Minkowski centered}. Note that $s(C)\in[1,n]$ for $C\in\K^n$, where $s(C)=1$ if and only if $C$ is centrally symmetric, while $s(C)=n$ if and only if $C$ is an $n$-dimensional simplex \cite{Gr}. Moreover, the Minkowski asymmetry $s:\K^n\rightarrow[1,n]$ is continuous w.r.t.~the Hausdorff metric (see \cite{Gr}, \cite{Sch} for some basic properties) and invariant under non-singular affine transformations.

The main contribution of this paper 
shows that the golden ratio is the largest asymmetry such that \eqref{eq:means_of_sets} can be optimal in the planar case.

\begin{thm}\label{thm:PlanarCase}
Let $C\in\mathcal K^2$ be Minkowski centered such that
\[
\left(\frac{C^\circ - C^\circ}{2}\right)^{\circ}\subset^{opt} \frac{C-C}{2},
\]
then $s(C)\leq \varphi$. 

Moreover, if $s(C)=\varphi$, there exists a non-singular linear transformation $L$ such that $L(C)=\GH:= \conv(\{p^1, \dots, p^5 \})$ with $p^1= (-1,-1)^T$, $p^2=(-1,0)^T$, $p^3=\left(0, \varphi \right)^T$, $p^4= (1,0)^T$, $p^5=(1,-1)^T$ is the \cemph{golden house}.
\end{thm}

\begin{figure}
  \begin{subfigure}[b]{0.47\textwidth}
    \centering
    \begin{tikzpicture}[scale=1.5]
      \draw [thick,dred] (-1,0) -- (0,1.618) -- (1,0) -- (1,-1) -- (-1,-1) -- (-1,0);
      \draw [thick, dblue] (-1.618,0) -- (0,-2.618) -- (1.618,0) -- (1.618,1.618) -- (-1.618,1.618) -- (-1.618,0);
      \draw [fill] node[above left]{$0$} (0,0) circle [radius=0.02];
      \draw [dashed] (-1, -2) -- (-1,2.1);
      \draw [dashed] (1, -2) -- (1,2.1);
      \draw (-1.3,-1) node {$p^1$};
      \draw (-1.28,0) node {$p^2$};
      \draw (0,1.9) node {$p^3$};
      \draw (1.0,-2.4) node {$-s(\GH)p^3$};
      \draw (1.28,0) node {$p^4$};
      \draw (1.3,-1) node {$p^5$};
      \draw (0.5,-0.75) node {$\GH$};
      \draw [fill] (0,-1) circle [radius=0.02];
      \draw (0,-1.27) node {$g$};
    \end{tikzpicture}
    \caption{$\GH$ (red), $-s(\GH) \GH$ (blue), and \\ parallel supporting halfspaces in $p^2$ \\ and $p^4=-p^2$ (dashed).} \label{fig:GH}
  \end{subfigure} \hfill
  \begin{subfigure}[b]{0.47\textwidth}
    \centering
    \begin{tikzpicture}[scale=1.5]
      \draw [thick, orange](-1,-1) -- (-1,1) -- (0,1.618) -- (1,1) -- (1,-1) -- (0,-1.618) -- (-1,-1);
      \draw [thick,dblue](-1,0) -- (-0.38195,1) -- (0.38195,1) -- (1,0) -- (0.38195,-1) -- (-0.38195,-1) -- (-1,0);
      \draw [thick, dred](-1,0.5) -- (-0.5,1.309) -- (0.5,1.309) -- (1,0.5) -- (1,-0.5) -- (0.5,-1.309)  -- (-0.5,-1.309) -- (-1,-0.5)-- (-1,0.5);
      \draw [thick, violet](-1,0) -- (-0.75,0.785) -- (0,1.2361) -- (0.75,0.785) -- (1,0) -- (0.75,-0.785) -- (0,-1.2361) -- (-0.75,-0.785) -- (-1,0);
      \draw [fill] (0,0) circle [radius=0.02];
      \draw [dashed] (-1, -2) -- (-1,2.1);
      \draw [dashed] (1, -2) -- (1,2.1);
      \draw (0,-2.618) node {};
		\end{tikzpicture}

    \caption{$\conv(\GH \cup (-\GH))$ (orange), $\frac{1}{2} (\GH-\GH)$ (red), $\left(\frac{1}{2} (\GH^\circ+(-\GH)^\circ) \right)^{\circ}$ (violet), and \\ $\GH \cap (-\GH)$ (blue).} \label{fig:GH-symm}
  \end{subfigure}
  \caption{The golden house and its symmetrizations.}
\end{figure}

The important facts about the construction of the golden house are the following:

\begin{enumerate}[(i)]
	\item $p^2=-p^4$,
	\item $\|p^2-p^3\|= \|p^4-p^3\|$,
	\item $\conv (\{p^1, -s(\GH)p^3,p^5 \})$ and $\conv (\{ p^2, p^3,  p^4 \})$ are similar up to reflection.
\end{enumerate}

Let $g:= [p^1,p^5] \cap [p^3,-s(\GH)p^3]$, $\alpha:=\|p^3-g\|$, and $\beta:=\|p^3\|$. 
Then we have on the one hand
\begin{equation}\label{eq:s1}
s(\GH)= \frac{ \|-s(\GH)p^3 \|}{\|p^3 \|}=\frac{ \|p^3-g \|}{\|p^3 \|}=  \frac{ \alpha}{\beta},
\end{equation}
and on the other hand
\begin{equation}\label{eq:s2}
s(\GH)= \frac{ \|-s(\GH)p^3 -p^3\|}{\|p^3 -g \|}= \frac{ \alpha+\beta}{\alpha}.
\end{equation}
Combining \eqref{eq:s1} and \eqref{eq:s2} we see that $s(\GH) = \varphi$.

To the best of our knowledge, this is the first explicit mentioning of a set with the properties of the golden house. Theorem \ref{lem:Charact_Opt_Means} shows that (i) and (ii) from above suffice to show that the in the case of the golden house (and its negative) optimal containment is reached in \eqref{eq:means_of_sets} throughout the full chain. Even more so, from Theorem 1.3 below it directly follows that the minimum is optimally contained in the maximum.

For any $C \in\K^n$ let $\bd(C)$ be the \cemph{boundary} of $C$ and for any $a\in\R^n \setminus \{0\}$
and $\rho\in\R$, $H^{\le}_{a,\rho} = \{x\in\R^n: a^Tx \leq \rho\}$ denote a \cemph{halfspace}.
We say that the halfspace $H^{\le}_{a,\rho}$ \cemph{supports} $C \in\K^n$ at $q \in C$, if $C \subset H^{\le}_{a,\rho}$ and $q \in \bd(H^{\le}_{a,\rho})$. 

\begin{thm}\label{lem:Charact_Opt_Means}
	Let $C \in \K^n$ be Minkowski centered. Then the following are equivalent:
	\begin{enumerate}[(i)]
		\item $C \cap (-C) \subset^{opt} \conv(C \cup (-C))$,
		\item $\left(\frac12 (C^\circ-C^\circ)) \right)^{\circ}\subset^{opt} \frac12 (C-C)$,
		\item there exist $p, -p \in \bd(C)$ and parallel halfspaces $H^{\le}_{a,\rho}$ and $H^{\le}_{-a,\rho}$ supporting $C$ at
      $p$ and $-p$, respectively.
	\end{enumerate}
\end{thm}

Let us mention that for any regular Minkowski centered $(2n+1)$-gon $P$ the
vertices of $-\frac{1}{s(P)} P$ are the midpoints of the edges of $P$. Hence, they obviously do not fulfill Part (iii) of Theorem \ref{lem:Charact_Opt_Means}. Letting $n$ grow, we see that there exist Minkowski centered $C \in \K^2$ with $s(C)$ arbitrary close to $1$ such that not all containments in the inequality chain \eqref{eq:means_of_sets} are optimal for $C$.
Furthermore, one may observe that a Minkowski centered regular pentagon has asymmetry $2/\varphi \approx 1.236 < \varphi$.





\section{Characterizations of optimal containment} 
\label{sec:OptimalityOtherMeans}

Let us first collect some simple set identities under affine transformations.

\begin{lemma}\label{lem:Means_Invariant}
Let $K,C \in \K^n$ and $A$ be a non-singular affine transformation. Then
\begin{equation*}
  \begin{split}
    A(K)\cap A(C)=A(K\cap C), \qquad \left( ((A(K))^\circ-(A(C))^\circ)/2 \right)^\circ=A\left((K^\circ- C^\circ)/2
    \right)^\circ, \\
    (A(K)+A(C))/2=A\left( (K+C)/2\right), \qquad \conv\left(A(K)\cup(A(C)\right)=A\left(\conv(K\cup C)\right).
  \end{split}
\end{equation*}
\end{lemma}


The following proposition characterizes the optimal containment $K\subset^{opt}C$ between two convex sets $K,C \in \K^n$
in terms of common boundary points and belonging supporting halfspaces (see \cite[Theorem 2.3]{BrK}).
\begin{proposition}\label{prop:Opt_Containment}
  Let $K,C\in\K^n$ and $K\subset C$. Then the following are equivalent:
  \begin{enumerate}[(i)]
  \item $K\subset^{opt}C$.
  \item There exist $k\in\{2,\dots,n+1\}$, $p^j\in K \cap \bd(C)$, $a^j$ outer normals of supporting halfspaces of
    $K$ and $C$ at $p^j$, $j=1,\dots,k$, such that
    $0\in\conv(\{a^1,\dots,a^k\})$.
  \end{enumerate}
  Moreover, in the case that $K,C\in\K^n_0$ (i) and (ii) are also equivalent to $K\cap \bd(C)\neq\emptyset$.
\end{proposition}

Lemma \ref{lem:Means_Invariant} together with Proposition \ref{prop:Opt_Containment} obviously yield the following corollary.

\begin{cor}\label{cor:invariant}
Let $C \in \mathcal{K}^n$ 
and let $L$ be a non-singular linear transformation. Then
\begin{enumerate}[a)]
  \item $C$ is Minkowski centered if and only if $L(C)$ is Minkowski centered. 
  \item $C \cap (-C) \subset^{opt} \conv(C \cup (-C))$
	if and only if 
	$L(C) \cap L(-C) \subset^{opt} \conv(L(C) \cup L(-C))$.
\end{enumerate}
\end{cor}

Let us now add a proposition that is a result of Klee \cite{K} 
reduced to the two-dimensional case.

\begin{proposition} \label{prop:klee}
Let $P,C \in  \K^2$, $P$ a polygon and $C$ 0-symmetric, such that $P \subset^{opt} C$. Then $0 \in P$.
\end{proposition}

Taking the two preceding propositions together we obtain the corollary below.

\begin{cor} \label{cor:zero-inside}
Let $C \in \K^2$ be Minkowski centered, but not 0-symmetric. Then there exist $p^1,p^2,p^3 \in \bd(C) \cap (-s(C)\bd(C))$ such that $0 \in \conv(\{p^1,p^2,p^3\})$.
\end{cor}

\begin{proof}
Let us first mention that the existence of two or three such touching points of $\bd(C) \cap (-s(C)\bd(C))$ is a direct consequence of Proposition \ref{prop:Opt_Containment} and if it would only be two it would follow $s(C)=1$.

Now, let $S$ be the intersection of the three common supporting halfsspaces of $C$ and $-s(C)C$ at the points $p^i$, $i=1,2,3$. In addition $C$ (together with $-1/s(C) C$) is also supported in $1/s(C)p^i$ by halfspaces with outer normals being the negatives of the outer normals of the starting three. Hence, we obtain that $\conv(\{p^1,p^2,p^3\})$ is optimally contained in the minimum $S \cap (-S)$ of $S$ and $-S$ and therefore, by Proposition \ref{prop:klee}, that $0 \in \conv(\{p^1,p^2,p^3\})$.
\end{proof}




\begin{proof}[Proof of Theorem \ref{lem:Charact_Opt_Means}]
  \text{}
  \begin{itemize}
  \item[(i) $\Rightarrow$ (ii)] This part of the proof follows directly from Proposition \ref{prop:means_of_sets}.
  \item[(ii) $\Rightarrow$ (iii)]
    Assuming that $(\frac{C^\circ-C^\circ}2)^\circ \subset^{opt} \frac{C-C}2$ we obtain from
    Proposition \ref{prop:Opt_Containment} that there exists a common boundary point $p$ of the two sets.
    Let $\rho_1, \rho_2>0 $ be the smallest factors such that $\frac{1}{\rho_1 } p \in \text{bd} (C)$ and
    $\frac{1}{\rho_2 } p \in \text{bd} (-C)$, respectively.
    On the one hand, this implies
    \[ \frac{1}{2} \left( \frac{1}{\rho_1} +  \frac{1}{\rho_2}  \right) p \in  \frac{C-C}{2} \]
    and since $p \in \bd \left( \frac{C-C}{2} \right)$, we have that
    \begin{equation}\label{eq:th12_1}
      1 \leq \left( \frac12\left(\frac{1}{\rho_1} +\frac{1}{\rho_2}\right)\right)^{-1}. 
    \end{equation}
    On the other hand, from $\frac{1}{\rho_1 } p \in \text{bd} (C)$ follows that
    $C^{\circ} \subset \{a \in \R^n: a^T p \leq \rho_1 \}$ and that there exists some $a^1 \in \bd (C^{\circ})$ such that
    $ (a^1)^T p = \rho_1$. Similarly, we obtain $-C^{\circ} \subset \{a \in \R^n: a^T p \leq \rho_2 \}$
    and the existence of $a^2 \in \bd (-C^{\circ})$ such that $ (a^2)^T p = \rho_2$.
    Hence, $\frac{1}{2} (C^{\circ} -C^{\circ} ) \subset \{a \in \R^n: a^T p \leq \frac{1}{2} (\rho_1 +\rho_2)\}$ and
    $\frac12(a^1+a^2) \in \bd \left( \frac{1}{2} \left( C^{\circ} -C^{\circ} \right) \right)$ with
    $\frac12(a^1+a^2)^T p = \frac{1}{2} (\rho_1 +\rho_2)$. This means
    \[	\frac{2}{\rho_1+\rho_2}p \in \bd \left( \frac{C^{\circ}-C^{\circ}}{2}\right)^{\circ},\]
    which by the fact that $p \in \text{bd} \left( \frac{C^{\circ}-C^{\circ}}{2}\right)^{\circ}$ implies
	\begin{equation}\label{eq:th12_2}
	\frac{2}{\rho_1+\rho_2} = 1.
	\end{equation}
	Combining \eqref{eq:th12_1}, \eqref{eq:th12_2}, we obtain that the arithmetic mean is not greater than the
  harmonic mean of $\rho_1$ and $\rho_2$, thus $\rho_1=\rho_2=1$.
  This proves $p\in C \cap (- C)$. 

  Finally, let $H_{a,\rho}^{\le}$  be a supporting half space of $(C-C)/2$ at $p$ and assume $H_{a,\rho}^{\le}$
  does not support $C$.
  Hence there would exist some $q \in C$ with $a^T q > \rho$. Now, since $p \in -C$, we obtain $(p+q)/2 \in (C-C)/2$, which, because of
	$a^T \left( \frac{p+q}{2} \right) >\rho$, contradicts the fact that $(C-C)/2\subset H_{a,\rho}^{\le}$.
  This proves $C \subset H_{a,\rho}^{\le}$ and analogously we one obtains $-C \subset H_{a,\rho}^{\le}$.
	However, the convexity of halfspaces now implies $\conv(C\cup(-C))\subset H_{a,\rho}^{\le}$, which shows that
  condition (iii) is fulfilled.

  \item[(iii) $\Rightarrow$ (i)]
    Assuming that $C$ is supported by $H_{a,\rho}^{\le}$, $H_{-a,\rho}^{\le}$ at $p$, $-p$, respectively,
    the same holds for $-C$. Hence, we have
    $p,-p\in C\cap(-C)$ and $\conv(C\cup(-C))$ is supported by $H_{a,\rho}^{\le}$, $H_{-a,\rho}^{\le}$ at $p,-p$, respectively.
    By Proposition \ref{prop:Opt_Containment} this means that $C \cap (-C) \subset^{opt} \text{conv} (C \cup (-C))$.
\end{itemize}
\end{proof}

\section{Main result}\label{sec:Main}



\begin{proof}[Proof of Theorem \ref{thm:PlanarCase}] 
Let $C\in\K^2$ be Minkowski centered, with $s:=s(C)>1$, such that $\left(\frac12 (C^\circ-C^\circ) \right)^{\circ}\subset^{opt} \frac12 (C-C)$. By Theorem \ref{lem:Charact_Opt_Means} this optimality condition is equivalent to $C\cap(-C)\subset^{opt}\conv(C\cup(-C))$ and to the existence of $-p,p \in \bd(-sC)$, as well as parallel 
halfplanes $-H, H$ supporting $-sC$ at $-p,p$, respectively. Since $C$ is Minkowski centered, we have $C \subset^{opt} -sC$ and therefore we obtain by Proposition \ref{prop:Opt_Containment} the existence of $k \in \{2,3\}$, $q^1,\dots,q^k \in \bd(C) \cap \bd(-sC)$ and outer normals of supporting halfplanes $a^1,\dots,a^k$ with $0 \in \conv(\{a^1,\dots,a^k\})$. Moreover, from $s>1$ easily follows that $k=3$.

It cannot be that $\pm p \not \in \{q^1,q^2,q^3\}$. Otherwise, let, e.g., $q^2=p$. Then we have $q^2 \in C \cap \bd(-sC)$ and thus  $-s p = -sq^2 \in \bd(-sC)$, which would imply $s(C)=1$. 

By Corollary \ref{cor:zero-inside} we have $0 \in \conv(\{q^1,q^2,q^3\})$. Hence, we can assume w.l.o.g.~that $q^1$ is located on one side, while $q^2, q^3$ on the other side of the line $\aff(\{-p,p\})$ and moreover even that $-p \in \pos\{q^1,q^3\}$ and $p \in \pos\{q^1,q^2\}$.
Observe that the lines $\aff (\{-p,q^3\})$ and $\aff (\{p,q^2\})$ intersect in some point $d^1$. Otherwise, we would have $q^3 \in \bd(-H)$ and $q^2 \in \bd(H)$ and therefore $[q^3,-p], [q^2,p] \subset \bd(-sC)$. This would imply that the segment $[-\frac{1}{s}  q^2,-\frac{1}{s} p]$, which is
parallel to $[q^3,-p]$, belongs to  $\bd(C) \cap \inter (-sC).$ Together with $q^1 \in \bd(C)$ and $s>1$ this would contradict the convexity of $C$.

Now, we choose $d^2 \in \bd(-H)$, $d^3 \in \bd(H)$ such that $q^1 \in [d^2,d^3]$ and $[d^2,d^3]$ is parallel to $[q^2,q^3]$.

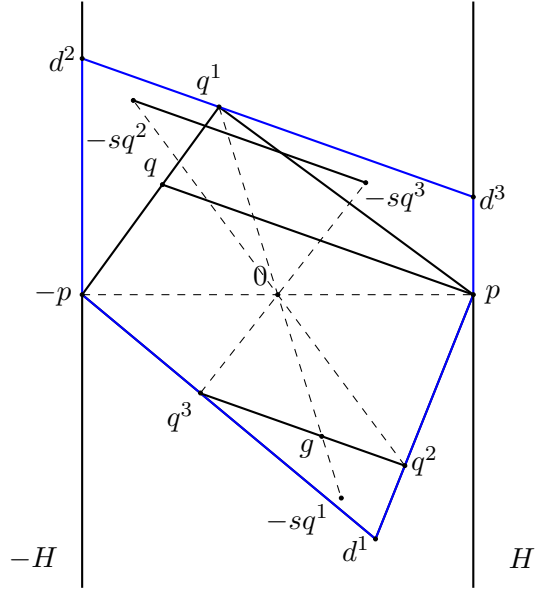
\begin{figure}
		\centering
		\begin{tikzpicture}[scale=1.3]
		\draw [thick] (2,0) -- (1,-2.5) -- (-2,0) -- (-0.6,1.921182) -- (2,0);
		\draw [thick,dblue] (2,0) -- (1,-2.5) -- (-2,0)  -- (-2,2.41712)  -- (2,1) -- (2,0);
		\draw [thick] (1.3,-1.75) -- (-0.8,-1.006);
		\draw [thick] (2,0) -- (-1.1792,1.12633);
		\draw [thick] (-2, -3) -- (-2,0);
		\draw [thick] (-2,2.41712) -- (-2,3);
		\draw [thick] (2, -3) -- (2,0);
		\draw [thick] (2, 1) -- (2,3);
		\draw [thick] (0.9,1.146) --(-1.4769,1.9882);
		\draw [dashed] (-2,0) -- (2,0);
		\draw [dashed] (-0.6,1.921182) --(0.65,-2.0812805);
		\draw [dashed] (-0.79,-1.006) --(0.9,1.146);
		\draw [dashed] (1.3,-1.75) --(-1.4769,1.9882);

		\draw [fill] node[above left]{$0$} (0,0) circle [radius=0.02];
		\draw [fill] (-2,0) circle [radius=0.02]; 
		\draw [fill] (2,0) circle [radius=0.02]; 
		\draw [fill] (1.3,-1.75) circle [radius=0.02]; 
		\draw [fill] (-0.79,-1.006) circle [radius=0.02]; 
		\draw [fill] (1,-2.5) circle [radius=0.02]; 
		\draw [fill] (-0.6,1.921182) circle [radius=0.02]; 
		\draw [fill] (2,1) circle [radius=0.02]; 
		\draw [fill] (-2,2.41712) circle [radius=0.02]; 
		\draw [fill] (0.65,-2.0812805) circle [radius=0.02]; 
		\draw [fill] (0.45,-1.45) circle [radius=0.02]; 
		\draw [fill] (-1.1792,1.12633) circle [radius=0.02]; 
		\draw [fill] (0.9,1.146) circle [radius=0.02]; 
		\draw [fill] (-1.4769,1.9882) circle [radius=0.02]; 

		\draw (-2.3,0) node {$-p$};
		\draw (2.2,0) node {$p$};
		\draw (0.8,-2.6) node {$d^1$};
		\draw (1.5,-1.7) node {$q^2$};
		\draw (-1.0,-1.25) node {$q^3$};
		\draw (-0.7,2.2) node {$q^1$};
		\draw (2.2,1) node {$d^3$};
		\draw (-2.2,2.41712) node {$d^2$};
		\draw (0.2,-2.3) node {$-sq^1$};
		\draw (0.3,-1.58) node {$g$};
		\draw (-1.3,1.3) node {$q$};
		\draw (-2.5,-2.7) node {$-H$};
		\draw (2.5,-2.7) node {$H$};
		\draw (1.2,1) node {$-sq^3$};
		\draw (-1.65,1.6) node {$-sq^2$};

		\end{tikzpicture}
	\caption{Construction used in the proof of Theorem  \ref{thm:PlanarCase}
	}
	\label{fig:Proof}
\end{figure}

Let us first prove that
\begin{equation}\label{eq:trianglefor-sc}
	-sq^1 \in \conv(\{q^2,q^3,d^1\}).
\end{equation}
Since $0 \in \conv(\{q^1,q^2,q^3\})$ 
we have $-sq^1 \in \pos(\{q^2,q^3\})$. 
Thus using the fact that  $q^2, q^3, -sq^1  \in \bd(-sC)$, the convexity of $-sC$ implies that $-sq^1 \in \conv(\{q^2,q^3,d^1\})$.

The next fact we want to see is
\begin{equation}\label{eq:trianglefor-sb}
	-sq^3 \in \conv(\{p,q^1,d^3\}).
\end{equation}
To see this, remember that $-H$ supports $-sC$ at $-p$. 
Moreover, directly from $-p \in \pos(\{q^1,q^3\})$ we see that $-sq^3 \in \pos(\{p,q^1\})$. However, since $p, q^1, -sq^3 \in \bd(-sC)$, the convexity of $-sC$ 
implies 
$-sq^3 \not \in \inter ( \conv(\{0,p,q^1\}))$. 
Finally, collecting the facts that $q^1,-sq^2, -sq^3 \in \bd(-sC)$, $q^1 \in \pos(\{-sq^2,-sq^3\})$, and the parallelity of $[-sq^2,-sq^3]$ and $[d^2,d^3]$, we obtain  
$-sq^3  \in \conv(\{p,q^1,d^3\})$.

Similarly to \eqref{eq:trianglefor-sb}, one may prove
\begin{equation}\label{eq:trianglefor-sa}
	-sq^2 \in \conv (\{-p,q^1, d^2 \}).
\end{equation}

Our goal is to determine the greatest possible $s$ such that $C\cap(-C) \subset^{opt} \conv(C\cup(-C))$ is still fulfilled. We say that 
the points $q^1,q^2,q^3$ present a \emph{valid situation} if they fulfill the conditions \eqref{eq:trianglefor-sc}, \eqref{eq:trianglefor-sb} and \eqref{eq:trianglefor-sa}. We do the following changes on $q^1,q^2,q^3$, so that after each step, we are still under a valid situation for the given asymmetry $s$:

\begin{enumerate}
    \item[(i)] Replace $q^2$ (resp. $q^3$) by the point in $[q^2,p]$ (resp. $[q^3,-p]$) such that $-sq^2\in -H$ (resp. $-sq^3\in H$).  
    Since $s>1$, $q^2$ belongs in the strip between $H$ and $-H$, and $-sp$ belongs outside the same strip and is closer to $-H$ than to $H$, then $-s[q^2,p]=[-sq^2,-sp]$ intersects $-H$ at a point $-s\widetilde{q}^2$. Let us replace $q^2$ by $\widetilde{q}^2$.
    \item[(ii)] Replace $q^1$ by $\mu q^1$, for some $\mu<1$, such that $\mu q^1\in[-sq^2,-sq^3]$.
    \item[(iii)] Substitute $q^1$ by $-\gamma d^1\in[-sq^2,-sq^3]$, for some $\gamma>0$.
\end{enumerate}
Recognize that $s\gamma d^1=-sq^1\in\conv(\{d^1,q^2,q^3\})$ implies $s\gamma\leq 1$. 

Now, we can study the maximal possible value for $s$, which means we want to characterize the situation, in which $s$ gets maximal such that $s\gamma\leq 1$. Thus we need to know the explicit value of $\gamma$ (depending on $s$). 

To do so, after a suitable linear transformation, suppose that $p=(1,0)$, and $H$ and $-H$ are vertical lines (perpendicular to $[-p,p]$). Because of Step (i) above we may furthermore assume $q^2=(1/s,-a)^T$ and $q^3=(-1/s,-1)^T$ for some $a \in (0,1]$. 
Now, we need the coordinates of $d^1$, which is the intersection of the lines $\aff\{p,q^2\}$ and $\aff\{-p,q^3\}$. We obtain
\[
d_2^1=-\frac{1}{1-\frac1s}(d_1^1+1)\quad\text{and}\quad d_2^1=\frac{a}{1-\frac1s}(d_1^1-1),
\]
resulting in 
\[
d^1=\left(\frac{a-1}{a+1},\frac{-2a}{(1-\frac1s)(a+1)}\right)^T.
\]
Now, we compute $\gamma$ such that Condition (iii) is fulfilled, i.e.
\[
-\gamma d^1\in[-sq^2,-sq^3]=[(-1,sa)^T,(1,s)^T]. 
\]
Hence, for some $\lambda\in[0,1]$, we have
\[
-\gamma\left(\frac{a-1}{a+1},\frac{-2a}{(1-1/s)(a+1)}\right)
=(1-\lambda)(-1,sa)^T+\lambda(1,s)^T=\left(-1+2\lambda,s((1-\lambda)a+\lambda )\right)^T
\]
and it is easy to check that this implies
\[
\gamma=\frac{(s-1)(a+1)^2}{4a-(s-1)(a-1)^2}.
\]
Thus the problem of finding the maximal $s$ under the condition $s\gamma\leq 1$ rewrites as
\[
\max s, \text{ s.t.~}\frac{s(s-1)(a+1)^2}{4a-(s-1)(a-1)^2}\leq 1.
\]
The above condition easily rewrites as 
\[
(s^2-1)(a+1)^2-4as \leq 0.
\]
We are interested in the maximum $s$, i.e., in the larger of the two roots of the equation $(s^2-1)(a+1)^2-4as = 0$, which computes to
\[
s=\frac{2a}{(a+1)^2}+\sqrt{1+\frac{4a^2}{(a+1)^4}}=:h(a),
\]
$a\in(0,1]$. Hence, the maximum of $s$ coincides with the maximum of $h(a)$ with $a\in(0,1]$. It is straightforward to verify that $h(a)$ is increasing in $(0,1]$, and thus, we can conclude that
\[
\max\,s=\max_{a\in(0,1]}\,h(a)=h(1)=\frac{1+\sqrt{5}}{2}=\varphi.
\]

Now, note that equality holds if and only if $a=1$, $\gamma=\varphi-1$, and $d^1=(0,-\varphi/(\varphi-1))^T$. Moreover, in the extreme case we have $\varphi\gamma=1$, which is true if and only if $-\varphi q^1=d^1$, 
$q^2=(1/\varphi,-1)^T$ and $q^3=(-1/\varphi,-1)^T$. 
Since $q^2\in\bd(-\varphi C)\cap[d^1,p]$, we have $[d^1,p]\subset\bd(-\varphi C)$. The same reasoning with $q^3$ replacing $q^2$ implies $[d^1,-p]\subset\bd(-\varphi C)$.
Moreover, $q^1=(0,1/(\varphi-1))^T = -\gamma d^1 \in [-\varphi q^2,-\varphi q^3]$. 
Thus $q^1\in\bd(-\varphi C)$ implies $[-\varphi q^2,-\varphi q^3]\subset\bd(-\varphi C)$. Since it is also clear that $[p,-\varphi q^3],[-p,-\varphi q^2]\subset\bd(-\varphi C)$, we got a complete description of the boundary of $-\varphi C$, thus proving
$-\varphi C=\conv(\{d^1,\pm p, -\varphi q^2,-\varphi q^3\})$. Finally, since $\varphi = 1/(\varphi-1)$ we obtain
\begin{align*}
C  = \conv\left(\left\{\left(0,\frac{1}{\varphi-1}\right)^T,\left(\pm\frac{1}{\varphi},0\right)^T,\left(\pm \frac{1}{\varphi}, -1 \right)^T\right\}\right) = \left(\begin{array}{cc}\frac{1}{\varphi}&0\\0& 1 \end{array}\right)\GH,
\end{align*}
which concludes the proof of our theorem.

\end{proof}

\begin{rmk}
	For every $s\in[1,\varphi]$ there exists $C\in\K^2$ Minkowski centered
	with $s(C)=s$ such that
\[
C\cap(-C)\subset^{opt}\conv(C\cup(-C).
\]
To see this we perform a symmetrization process: making a hexagon from the pentagon $\GH$ by adding the point $(0,-\tau)^T$ for $\tau \in [1,\varphi^2]$ and translating the whole set in direction of $(1,0)^T$ such that it is Minkowski centered again. This way we obtain continuously monotonly shrinking Minkowski asymmetry with growing $\tau$, ending in a 0-symmetric hexagon when $\tau=\varphi^2$, while keeping property (iii) of Theorem \ref{lem:Charact_Opt_Means} true. 
	
\end{rmk}

\end{document}